\documentclass[12pt]{amsart}
\usepackage{amsthm}
\usepackage{amsfonts,amssymb,amsmath}
\theoremstyle{definition}
\newtheorem{theorem}{Theorem}
\begin{document}
\title{Symplectic and orthogonal tableaux revisited}
\author{William M. McGovern}
\begin{abstract}
Generalizing a construction of Berele for symplectic groups, we give a uniform construction of irreducible polynomial representations of classical groups, including spin groups, using semistandard tableaux.  We also give an explicit decomposition of the homogeneous coordinate ring of the flag variety for classical groups and explicit generators for the ideal of functions vanishing on this variety.  
\end{abstract}
\maketitle

\section{Introduction}
It is well known that irreducible polynomial representations of $GL_n\mathbb C$ can be modelled by semistandard Young tableaux with entries in $[n]=\{1,\ldots,n\}$.  The aim of this paper is to generalize this construction to the symplectic and orthogonal groups Sp$_{2n}\mathbb C$ and $O_n\mathbb C$.  We will define the notions of symplectic and orthogonal tableaux; the symplectic case is also well known (see \cite{KE83,B86}), while our orthogonal tableaux differ from the usual ones in the literature (see \cite{S90,S90',KE83}).  Following Fulton's treatment of the general linear case \cite{F97} we will define an action of Sp$_{2n}\mathbb C$ (resp.\ $O_n\mathbb C$) on the complex vector space spanned by symplectic (resp.\ orthogonal) tableaux, in such way that these tableaux are weight vectors for the corresponding representations, with weights attached to tableaux as in \cite{KE83}.  Thus we will exhibit a very tight connection between the tableaux and the representations, mirroring this connection in the general linear case.  Moreover, we will give an explicit decomposition of the homogeneous coordinate ring of the flag variety of a symplectic or orthogonal group as the direct sum of its irreducible polynomial representations, again as for a general linear group, along the way showing that the vanishing ideal of this flag variety is generated by quadratic polynomials.  In the last section we will define the appropriate kind of tableaux for genuine polynomial representations of the Pin groups Pin$_n\mathbb C$ and extend our results to those groups.

\section{Semistandard tableaux and representations of $GL_n\mathbb C$}

We begin by briefly reviewing semistandard tableaux and polynomial representations of $GL_n\mathbb C$, following Chapter 8 of \cite{F97}.   The key definition applies in a broader context.  Given a partition $\lambda=(\lambda_1,\ldots,\lambda_m)$ of an integer $r$, a commutative ring $R$, and an $R$-module $E$, we write the decomposable tensors in the tensor power $T^rE=\otimes^rE$ as Young tableaux $T$ of shape $\lambda$ with entries in $E$.  Define the {\sl Schur module} $E^\lambda$ to be the quotient of $T^rE$ by the {\sl alternating relations} $T=-T'$, where $T'$ is obtained from $T$ by interchanging two entries in one column, and the {\sl exchange relations} $T=\sum_i T^i$, where the sum runs over all tableaux $T_i$ obtained from $T$ by interchanging the top $k$ elements in one fixed column of $T$ with any $k$ elements in another fixed column lying to the left of the first one, preserving the vertical order of the chosen elements throughout.  Here $k$ is any positive integer less than or equal to the length of the rightmost chosen column.  For example, if $\lambda=(r)$ has just one part, then $E^\lambda$ identifies with the symmetric power $S^rE$, while if $\lambda=(1,\ldots,1)$ has $r$ parts all equal to 1, then $E^\lambda$ identifies with the exterior product $\wedge^rE$.  In general, if $G$ is a group of invertible $R$-endomorphisms of $E$, then it is easy to see that the action of $G$ on $T^rE$ descends to $E^\lambda$.

Now we specialize to the case $R=\mathbb C,E=\mathbb C^n,G=GL_n\mathbb C$.  Then $E^\lambda$ coincides with the unique irreducible polynomial representation $V^\lambda$ of $GL_n\mathbb C$ with highest weight $\lambda$ \cite[Theorem 2, p. 114]{F97}; since such representations are well known to correspond bijectively to their highest weights, which are partitions $\lambda=(\lambda_1,\ldots,\lambda_n)$ with exactly $n$ parts (allowing 0 as a part). This construction gives all the irreducible polynomial representations of $GL_n\mathbb C$.  Here $E^\lambda$ has as a basis the set of semistandard tableaux $T$ on $[n]$ of shape $\lambda$, that is, Young diagrams of shape $\lambda$ and entries in $[n]$, arranged so that entries weakly across rows and strictly down columns \cite[Theorem 1, p. 110]{F97}.  An entry $i$ in a semistandard tableau is identified with the $i$th standard basis vector $e_i$ of $\mathbb C^n$.  The {\sl weight} $w(T)$ of such a tableau $T$ is the monomial $\prod_{i=1}^n x_i^{a_i}$ (denoted $(a_1,\ldots,a_n)$ for short), where $a_i$ is the number of times $i$ appears in $T$; note that the diagonal matrix with diagonal entries $x_1,\ldots,x_n$ acts on $T$ by multiplying by this scalar.  An arbitrary tableau of shape $\lambda$ with entries in $[n]$ can then be written as a unique combination of semistandard tableaux using the straightening algorithm implicit in the proof of \cite[Theorem 1, p. 110]{F97}.  The homogeneous prime ideal corresponding to the flag variety of $GL_n\mathbb C$ is then generated by quadratic polynomials corresponding to the defining relations of the modules $E^\lambda$ as $\lambda$ runs over all partitions $(\lambda_1,\ldots,\lambda_n)$, so that the corresponding coordinate ring is isomorphic as a $G$-module to the direct sum of the $E^\lambda$ \cite[Proposition 2, p. 126, and Proposition 1, p. 135]{F97}. 

\section{Symplectic Schur modules and tableaux}

In this section, following Berele \cite{B86}, we give an analogous construction of the irreducible polynomial representations of a symplectic group.  Let $B=\{e_1,\ldots,e_n,e_{\bar 1},\ldots,e_{\bar n}\}$ be a basis of $E=\mathbb C^{2n}$. Define a nondegenerate skew form $f=(\cdot,\cdot)$ on $E$ via $f(e_i,e_j)=f(e_{\bar i},e_{\bar j})=0,\hfil\break f(e_i,e_{\bar j})=\delta_{ij}$.  The isometry group $G_n=\,$Sp$_{2n}\mathbb C$ of $f$ may be identified with the stabilizer of the alternating tensor $a=\sum_{i=1}^n (e_i\otimes e_{\bar i}- e_{\bar i}\otimes e_i)\in\bigotimes^2E$ in $GL(E)$.  Let $\lambda=(\lambda_1,\ldots,\lambda_n)$ be a partition with $n$ parts and let $E^\lambda$ be the Schur module for $GL_{2n}\mathbb C$ corresponding to $\lambda$, defined using tableaux with entries in $[n]'=\{1,\ldots,n,\bar 1,\ldots,\bar n\}$.  Given a tableau $T$ with shape $\lambda$ and two boxes $B,B'$ of $T$. let $F$ be the (partial) filling obtained from $T$ by erasing the entries in $B$ and $B'$. We impose the {\sl symplectic relation}
\begin{equation}
S_F=\sum_{i=1}^n (F_i-F_i') =0
\end{equation}
\noindent on $E^\lambda$, where $F_i$ is the tableau obtained from $F$ by inserting the entries $i$ and $\bar i$ into $B$ and $B'$, respectively, while $F_i'$ is obtained from $F$ by inserting the entries $\bar i$ and $i$ in $B$ and $B'$.  Here an entry $i$ in a box is identified with the basis vector $e_i$, while an entry $\bar i$ is identified with $e_{\bar i}$.  One checks that applying an alternating or exchange relation to a symplectic relation gives a linear combination of symplectic relations.  The {\sl symplectic Schur module} $E^{\lambda,f}$ corresponding to $\lambda$ and the skew form $f$ is defined to be the quotient of $E^\lambda$ by these relations.  Since the action of $G_n$ on $\bigotimes^2E$ preserves the alternating tensor $a$, it is easy to check that there is a well-defined action of $G_n$ on $E^{\lambda,f}$.  

We now order the indices $i,\bar i$ via $1<\bar 1<2<\bar 2<\cdots<n<\bar n$ and call a tableau $T$ with entries in $[n]'$ semistandard if its entries increase weakly across rows and strictly down columns.  The weight $w(T)$ of $T$ is the monomial $\prod_{i=1}^n x_i^{b_i}$, or just $(b_1,\ldots,b_n)$, where $b_i$ is the number of times that $i$ occurs in $T$, minus the number of times $\bar i$ occurs in $T$.  A typical diagonal matrix in $G_n$ has diagonal entries $x_1,x_1^{-1},\ldots,x_n,x_n^{-1}$, writing the matrices in $G'$ with respect to the basis $B$; this matrix acts on $T$ by multiplying it by the scalar $w(T)$.  We say that a semistandard tableau is {\sl symplectic} if the entries in its $i$th row are all greater than or equal to $i$ in this ordering; we call this condition the {\sl symplectic condition}.  

\begin{theorem}
With notation as above, a basis for $E^{\lambda,f}$ is given by the symplectic tableaux of shape $\lambda$.  This representation coincides with the irreducible one $V_\lambda$ of highest weight $\lambda$.  The vanishing ideal of the flag variety for $G_n$ is generated by quadratic polynomials corresponding to the alternating, exchange, and symplectic relations.
\end{theorem}

\begin{proof}
If a semistandard tableau $T$ is not symplectic and the highest entry in the leftmost column violating the symplectic condition is the $i$th one, then this entry is necessarily $\overline{i-1}$, the entry immediately above it is $i-1$, and  for all $j\le i-1$, exactly one of the entries $j$ or $\bar j$ appears in this column above the $i$th entry.  Applying the symplectic relation to the filling obtained from $T$ by erasing the entries $i-1$ and $\bar{i-1}$ in this column and using the alternating relations, we rewrite $T$ as a combination of tableaux lying higher in the total ordering of fillings of shape $\lambda$ occurring the proof of Theorem of \cite{F97}; using the straightening algorithm in that proof, we rewrite $T$ as a combination of semistandard tableaux, each again higher than $T$ in the ordering on fillings.  As in that proof, we conclude that symplectic tableaux of shape $\lambda$ span $E^{\lambda,f}$.

We show that symplectic tableaux of shape $\lambda$ are independent and that $E^{\lambda,f}$ is irreducible with highest weight $\lambda$ by induction on $n$.  The case $n=1$ is clear, since if $\lambda$ has only one part $m$, then all symplectic relations are consequences of the exchange relations and the $V_\lambda$ is just the $m$th symmetric power of the standard representation, which is well known to be irreducible.  In general, if $\lambda=(\lambda_1,\ldots,\lambda_n)$, when the restriction of $V_\lambda$ to the next lower symplectic group $G_{n-1}$ is known to be $\sum m_\mu V_\mu$, where $\mu=(\mu_1,\ldots,\mu_{n-1})$ runs over all partitions such that
\vskip .1in
\begin{equation}
\lambda_1\ge\mu_1\ge\lambda_2\ge\cdots\ge\mu_{i-1}\ge\lambda_i\ge\lambda_{i+1}>\mu_i\ge\lambda_{i+2}\ge\cdots\ge\mu_{n-1}\ge0;
\end{equation}
\vskip .1in
\noindent for some $i$ between 1 and $n$ \cite{Z73}.  Here
\vskip .1in
\begin{equation}
m_\mu=(\lambda_1-\mu_1+1)\cdots(\lambda_{i-1}-\mu_{i-1}+1) (\lambda_i-\lambda_{i+1}+1)\cdots(\mu_{n-1}-\lambda_{n+1}+1)
\end{equation}
\vskip .1in
\noindent if $\mu$ is as in $(2)$; we take $\lambda_{n+1}$ to be 0.  All boxes with $n$ or $\bar n$ occur at the ends of rows strictly longer than the rows just below them, the entries $n$ in such rows must precede all the entries $\bar n$, and all $n$s lie above all $\bar n$s in their columns.  The desired independence of the symplectic tableaux and the irreducibility of $E^{\lambda,f}$ follow from the inductive hypothesis and the branching rule.  (It does no harm to omit the term $F_n-F_n'$ in any symplectic relation when studying the action of $V_{n-1}$ on symplectic tableaux of shape $\lambda$, since $G_{n-1}$ preserves the span of the $e_i$ and $e_{\bar i}$ for $i\le n-1$ in $E$).

It is well known that the flag variety of $G_n$ may be identified with the set of complete isotropic flags $E_0\subset E_1\subset\cdots\subset E_n$ of $E$ with respect to $f$.  The alternating and exchange relations are quadratic polynomials and define the partial flag variety $\mathcal F$ of all flags\hfil\break $E_0'\subset\cdots\subset E_n'$ of $E$ with $\dim E_i=i$ for all $i$.  Now given any vector space $V$ equipped with a skew form $\langle\cdot,\cdot\rangle$, it is well known that this form induces a skew form on any exterior power $\wedge^kV$ of $V$; the induced form is nondegenerate if and only if the original one is, and identically 0 if and only if the original one is (for $0<k<\dim V$).  Thus the symplectic relations amount to a quadratic polynomial equations cutting out that portion of $\mathcal F$ identifying with the flag variety of $G_n$.  Moding out by all these relations, we get the direct sum of all representations $V^\lambda$ of $G_n$, each occurring once, which by Frobenius reciprocity and highest weight theory coincides as a representation of $G_n$ with the coordinate ring of the flag variety of $G_n$.  Hence the alternating, exchange, and symplectic relations generate the vanishing ideal of this variety.
\end{proof}

The above proof avoids the reference in \cite{B86} to \cite{KE83} to prove the linear independence of symplectic tableaux.  For some properties of the Schur polynomials corresponding to the irreducible representations of $G_n$ see \cite{S90,P94}.

\section{Orthogonal Schur modules and tableaux}
Let $B=\{e_1,\ldots,e_n,e_{\bar 1},\ldots,e_{\bar n},e_0\}$ be a basis of $E=\mathbb C^{2n+1}$; similarly let $B'=\{e_1,\ldots,e_n,e_{\bar 1},\ldots,e_{\bar n}\}$ be a basis of $E'=\mathbb C^{2n}$.  Define a nondegenerate symmetric form $f=(\cdot,\cdot)$ on $E$ via $f(e_i,e_j)=f(e_{\bar i},e_{\bar j})=0,\hfil\break f(e_i,e_{\bar j})=\delta_{ij},f(e_0,e_0)=1$,.  Let $f'$ be the corresponding symmetric form on $E'$.  The isometry group $H_n=O_{2n+1}\mathbb C$ of $f$ may be identified with the stabilizer of the symmetric tensor\hfil\break $s=e_0\otimes e_0+\sum\limits_{i=1}^n (e_i\otimes e_{\bar i}+e_{\bar i}\otimes e_i)\in\bigotimes^2E$ in $GL(E)$; similarly the orthogonal group $H_n'=O_{2n}\mathbb C$ identifies with the stabilizer of $s'=\sum\limits_{i=1}^n (e_i\otimes e_{\bar i}+e_{\bar i}\otimes e_i)\in\bigotimes^2E'$.  Let $\lambda=(\lambda_1,\ldots,\lambda_n)$ be a partition with $n$ parts and let $E^\lambda$ be the Schur module for $GL_{2n+1}\mathbb C$ corresponding to $\lambda$, using tableaux with entries in $[n]''=[n]'\cup\{0\}$; similarly let $E^{\lambda'}$ be the Schur module for $GL_{2n}\mathbb C$ corresponding to $\lambda$, using tableaux with entries in $[n]'$.  As in the symplectic case, start with a tableau $T$ and erase the entries in two boxes $B,B'$ in $T$, obtaining a partial filling $F$.  Impose the {\sl orthogonal relation}
\begin{equation}
O_F=F_0+\sum_{i=1}^n (F_i+F_i')=0
\end{equation}
\noindent on $E^\lambda$. Here $F_i$ (resp.\ $F_i'$) is obtained from $F$ by inserting the entries $i$ and $\bar i$ (resp.\ $\bar i$ and $i$) in $B$ and $B'$, while $F_0$ is obtained from $F$ by putting 0 entries in both $B$ and $B'$.  As above, an entry $i$ is identified with the basis vector $e_i$ for $0\le i\le n$, while an entry $\bar i$ is identified by $e_{\bar i}$ for $1\le i\le n$.   Similarly, we impose the corresponding relation $O_F=0$ on $E^{\lambda'}$, omitting the $F_0$ term.  Thanks to the alternating relations, these relations hold automatically if $B,B'$ lie in the same column of $T$ and it is enough to impose them in the cases where $B,B'$ lie in the same row.  The {\sl orthogonal Schur module} $E^{\lambda,f}$ is defined to be the quotient of $E^\lambda$ by these relations; similarly the orthogonal Schur module $E^{\lambda',f'}$ is the quotient of $E^{\lambda'}$ by the corresponding relations.  The groups $I_n,I_n'$ act on $E^{\lambda,f},E_{\lambda',f'}$, respectively.  

We order the indices via $1<\bar 1<\cdots<n<\bar n<0$.  The weight $w(T)$ of a semistandard tableau $T$ is defined as in the symplectic case, ignoring occurrences of 0 in $T$.  We will construct explicit bases for $E^{\lambda,f}$ and $E^{\lambda',f'}$ below.  

Recall that irreducible polynomial representations of $H_n$ are indexed in a two-to-one fashion by partitions $\lambda$ as above, the two representations $V_\lambda,V_{\lambda'}$ corresponding to $\lambda$ differing only in that the center of $H_n$ (which has order 2) acts trivially on $V_\lambda$ but not on $V_{\lambda'}$.  Irreducible polynomial representations of $H_n'$ are also indexed by partitions $\lambda=(\lambda_1,\ldots,\lambda_n)$, but this time there is only one representation $W_\lambda$ attached to $\lambda$ if $\lambda_n\ne0$, while there are two, say $W_\lambda$ and $W_{\lambda'}$, if $\lambda_n=0$.  (The label is determined by decreeing that the representation $W_\lambda$ of $H_n'$ occur in the representation $V_\lambda$ of $H_n$ if $\lambda_n=0$.)  If $\lambda_n\ne0$, then the representation $W_\lambda$ has two highest weights, namely $\lambda$ and $\lambda^-=(\lambda_1,\lambda_2,\ldots,-\lambda_n)$; otherwise $W_\lambda,W_\lambda'$ both have $\lambda$ as their unique highest weight. 

\begin{theorem}
The modules $E^{\lambda,f}$ and $E^{\lambda',f'}$ of $H_n$ or $H_n'$ are irreducible with highest weight $\lambda$; more precisely, $E^{\lambda,f}\cong V_\lambda$ if the sum $|\lambda|$ of the parts of $\lambda$ is even and $E^{\lambda,f)}\cong V_\lambda'$ otherwise.  Similarly for $E^{\lambda',f'}$.  The vanishing ideals of the flag varieties of $H_n$ and $H_n'$ are generated by quadratic polynomials corresponding to the alternating, exchange, and orthogonal relations.
\end{theorem}

\begin{proof}
By induction on the dimension $2n+1$ or $2n$ of the ambient space $E$ or $E'$.  in the base case $n=1$ the partition $\lambda$ has a single part $m$.  For $H_1'$ the orthogonal and exchange relations imply that $E^{\lambda',f'}$ has as a basis the two tableaux with all 1s and all $\bar 1$s, having the respective weights $m$ and $-m$.  For $H_1$, the orthogonal relations show that tableaux with two 0s are linear combinations of other tableaux, so a basis of $E^{\lambda,f}$ is given by the tableaux with entries $1,\bar 1,$, and $0$, with at most one 0 and the $\bar1$s occurring to the right of the $1$s.  The weights are thus $m,m-1,\ldots,-m$, and the first assertion holds in both cases.  Assume now that the result holds for $H_n'$ and recall the branching law from $H_n$ to $H_n'$ \cite{Z73}.  The restriction of $V_\lambda$ to $H_n'$ is then $\sum_\mu W_\mu$, where $\mu$ runs over all partitions $(\mu_1,\ldots,\mu_n)$ such that
\begin{equation}
\lambda_1\ge\mu_1\ge\lambda_2\ge\mu_2\ge\cdots\ge\lambda_n\ge\mu_n
\end{equation}
\noindent The restriction of $V_{\lambda'}$ to $I_n'$ is the same, except that we replace $W_\mu$ on the right side by $W_{\mu'}$ if $\mu_n=0$.  Let $B_n'$ be the stabilizer in $I_n'$ of the maximal isotropic flag of $E'$ whose $i$th subspace $E_i'$ is spanned by $e_1,\ldots,e_i$.  The inductive hypothesis shows that the only $B_n'$-highest weight vectors in $E^{\lambda,f}$ are multiples of semistandard tableaux $T$ whose $i$th row has all nonzero entries equal to $i$.  The branching law then shows that $E^{\lambda,f}$ decomposes over $I_{n-1}$ in the same way as $V_\lambda$ does (as in the symplectic case it does no harm to omit the term $F_0$ from all orthogonal relations when studying the $I_n'$ action).  Hence the copy of $V_\lambda$ in $E^{\lambda,f}$ generated by the tableau $T_\lambda$ whose $i$th row has all entries equal to $i$ fills out $E^{\lambda,f}$ .  The other assertion follows as in the symplectic case.

Now assume that the result holds for $H_{n-1}$ and again let $\lambda=(\lambda_1,\ldots,\lambda_n)$.  We change the basis $B'$ of the ambient space $E'$, replacing $e_n$ and $e_{\bar n}$ by $e_s=e_n+e_{\bar n}$ and $e_d = e_n-e_{\bar n}$, respectively, denoting these vectors by $s$ and $d$ whenever they appear as entries in tableaux and decreeing that $\overline{n-1}<s<d$.   Let $E''$ be the span of $e_i,e_{\bar i}$, and $e_s$ for $i\le n-1$.  We identify $H_{n-1}$ with the stabilizer of $e_d$ in $H_n'$.  Replace $s'$ by $\tilde s$, the symmetric tensor obtained from it by replacing the last two terms $e_n\otimes e_{\bar n}+e_{\bar n}\otimes e_n$ by $e_s\otimes e_s+e_d\otimes e_d$.  Change the orthogonal relations accordingly, using $s$ and $d$ instead of $n$ and $\bar n$.  This does not change the structure of the module $E^{\lambda',f'}$.  Note that the entries $s$ and $d$ make no contribution to the weight of a semistandard tableau relative to $I_{n-1}$.  Now recall the branching law from $H_n'$ to $H_{n-1}$ \cite{Z73}.  It asserts that $W_\lambda$ decomposes over as $\sum_\mu (V_\mu+V_\mu')$, where $\mu=(\mu_1,\ldots,\mu_{n-1})$ runs over all partitions with
\begin{equation}
\lambda_1\ge\mu_1\ge\lambda_2\ge\mu_2\ge\cdots\mu_{n-1}\ge\lambda_n
\end{equation}
\noindent if $\lambda_n\ne0$; if $\lambda_n=0$, then one replaces $(V_\mu+V_{\mu'})$ above by $V_\mu$.  A similar rule holds for $W_{\lambda'}$ if $\lambda_n=0$.  The stabilizer $\tilde B_{n-1}$ in of the maximal isotropic flag of $E''$ whose $i$th subspace $E_i''$ is spanned by $e_1,\ldots,e_i$ for $i\le n-1$ is a Borel subgroup of $I_{n-1}$.  The only $I_{n-1}$-highest weight vectors in $E^{\lambda',f'}$ are multiples of semistandard tableaux such that every entry of the $i$th row is $i$, except for entries at the bottom of their columns, which are allowed to be $d$ if $i<n$ and $d$ or $s$ if $i=n$.  Applying the orthogonal relations to fillings obtained by erasing entries of boxes in the $n$th row (if there is one), we find that any such $T$ with entries $d$ and $s$ both appearing in the $n$th row are equal to 0 in $E^{\lambda',f'}$.  It follows that $E^{\lambda',f'}$ decomposes over $I_{n-1}$ in the same way as $W_\lambda$ does, whence the copy of $W_\lambda$ appearing in it (as in the previous case) fills out the entire module $E^{\lambda',f'}$.  The other assertion again follows as in the symplectic case.  Finally, the assertion about the prime labels may be checked directly.
\end{proof}

We say that a semistandard tableau $T$ with entries in $[n]'$ or $[n]''$ satisfies the {\sl parity condition} if whenever entries $i$ and $\bar i$ appear in the $i$th row, there is an entry $i$ in the next higher row immediately above the entry $\bar i$.  We say that a semistandard $T$ is {\sl quasi-symplectic} if it is obtained from a tableau $T'$  satisfying the symplectic condition by adding boxes $B_1,B_2$ labelled $i$ and $\bar i$ in the $i$th and $(i+1)$st positions of the first column, respectively, possibly together with other boxes with entries larger than $\bar i$, such that if the $B_i$ and all boxes in any column with entries larger than $\bar i$ are removed, then the resulting tableau has the shape of a Young diagram with no retained box immediately to the right of a removed one.  For example, the tableau
\vskip .1in
\[
\begin{matrix} 1&1\\ \bar 1& \end{matrix}
\]
\vskip .1in
\noindent is not quasi-symplectic, since removing the boxes in the first column gives a tableau whose unique box lies immediately to the right of a removed one.  Likewise the tableau
\vskip .1in
 \[
\begin{matrix} 1&2\\2& \bar 2\\\bar 2& \end{matrix}
\]
\vskip .1in
\noindent is not quasi-symplectic, but it becomes quasi-symplectic if $2,\bar 2$ in the rightmost column are replaced by $3,\bar 3$, respectively.  Call a tableau {\sl orthogonal} if it is semistandard, quasi-symplectic, and satisfies the parity condition.  Then we have

\begin{theorem}
The module $E^{\lambda,f}$ of $I_n$ has as a basis the orthogonal tableaux of shape $\lambda$ with entries in $[n]''$.  The module $E^{\lambda',f'}$ has as basis the set of orthogonal tableaux of shape $\lambda$ with entries in $[n]'$.
\end{theorem}
\begin{proof}
By induction on the dimension of $E$ or $E'$.  Assume that the result holds for $H_n'$.  For every fixed choice of the boxes with 0 entries in an orthogonal tableau of shape $\lambda$ have been fixed, the orthogonal tableaux of this shape with 0 entries in exactly those boxes provide a basis for an irreducible $H_n'$-module of highest weight $\mu$, where the diagram of $\mu$ is obtained from that of $\lambda$ by removing these boxes.  Moreover, the partitions $\mu$ arising in this way are exactly those corresponding to the irreducible $H_n'$-constituents of either irreducible $I_n$-module of highest weight $\lambda$, by the branching law from $H_n$ to $H_n'$.  Thus the result holds for $H_n$.

Now assume that the result holds for $H_{n-1}'$.  The branching law from $H_n'$ to $H_{n-1}'$ may be obtained by concatenating the branching law from $H_n'$ to $H_{n-1}$ and that from $H_{n-1}$ to $H_{n-1}'$.  It is the same as that for $G_n$ to $G_{n-1}$ in the symplectic case given above, replacing $V_\lambda,V_\mu$ by $W_\lambda,W_\mu$, except in two respects.  The final factor $f_n:=\mu_{n-1}-\lambda_{n+1}+1$ in the multiplicity of $W_\mu$ is replaced by $\min(f_n,2)$ (recall that $\lambda_{n+1}=0$ here); and the single module $W_\mu$ is replaced by the sum $W_\mu+W_\mu'$ whenever $\mu_{n-1}=0$.  Fix a choice $C$ of the boxes with entries $n$ and $\bar n$ in an orthogonal tableau of shape $\lambda$ and let $\mu$ be the shape of the diagram obtained from that of $\lambda$ by removing those boxes.  Given an orthogonal tableau of shape $\mu$ with entries in $[n-1]'$, the number of ways to fill the boxes of $C$ with $n$ or $\bar n$ equals the multiplicity of $W_\mu$ in $W_\lambda$, {\sl unless} the last part $\mu_{n-1}$ of $\mu$ is 0.  In that case, allowing boxes in $C$ lying in the first column to be filled by $n-1$ and $\overline {n-1}$ rather than $n$ and $\bar n$ under the quasi-symplectic condition accounts for the replacement of $V_\mu$ in the symplectic branching rule by $W_\mu+W_{\mu}'$.   Similarly, the parity condition accounts for the last factor in the multiplicity of $W_\mu$ being $\min(f_n,2)$ rather than $f_n$; recall that there is no parity condition in the symplectic case.  Thus the result holds for $I_n'$, as desired.
\end{proof}

The above bases correct the mistaken ones in \cite{KE83}.  We remark that another tableau basis of the representation $E^{\lambda,f}$ of $H_n$ has been given by Sundaram in \cite{S90}.  She defines a tableau with entires in $[n]''$ of shape $\lambda$ to be orthogonal if the entries in each row increase weakly and the entries in each column increase strictly, except that multiple 0 entries are allowed in a column, but not in a row.  

\begin{theorem}{\cite{S90}}
With this definition of orthogonality, orthogonal tableaux of shape $\lambda$ form a basis for $E^{\lambda,f}$ such that for every weight $\mu$, the number of such tableaux of weight $\mu$ equals the multiplicity of the $\mu$-weight space in $E^{\lambda,f}$.
\end{theorem}

Unfortunately, there is no natural action of $H_n$ on the span of such tableaux.  Nor does a comparable basis exist for representations of $H_n'$.  Other bases are given in terms of Gelfand patterns in \cite{P94}.

\section{Schur modules for Pin groups and Pin tableaux}
We conclude by treating the genuine polynomial representations of the Pin groups $P_n$ and $P_n'$, the respective simply connected double covers of $O_{2n+1}\mathbb C$ and $O_{2n}\mathbb C$; that is, the polynomial representations of these groups not descending to the orthogonal groups.  Define the basis $B$ or $B'$ of the ambient vector space $E=\mathbb C^{2n+1}$ or $E'=\mathbb C^{2n}$ and the form $f$ or $f'$ as in the orthogonal case.  First we look at the spin representation $\Sigma_n$, having $(1/2,\ldots,1/2)$ as a highest weight and $\{(\pm(1/2),\ldots,\pm(1/2))\}$ as the set of all of its weights. $\Sigma_n$ admits a unique module structure over $P_n'$ and two such structures over $P_n$; in all cases it is irreducible \cite{HP06}.  A basis is given by the set of tableaux $T$ consisting of a single column of $n$ half-boxes, the $i$th of them with entry $a_i=i$ or $\bar i$ (so that the dimension of this representation is $2^n$).  The weight of $T$ has $i$th coordinate $1/2$ if $a_i=i$ and $-1/2$ otherwise.  

In general a genuine representation $V_v$ of $P_n$ or $P_n'$ corresponds to an $n$-tuple $v=(v_1,\ldots,v_n)$ such that $\lambda=(\lambda_1,\ldots,\lambda_n)$ is a partition, where $\lambda_i=v_i-(1/2)$.  Following \cite{KE83}, we decree that a diagram $D$ of shape $v$ consists of a single column of $n$ half-boxes added to the left of an ordinary Young diagram of shape $\lambda$.  A tableau of this shape is obtained by filling the half-boxes as for $\Sigma_n$ and the remaining boxes as for $H_n$ or $H_n'$ (so using entries from $[n]'\cup\{0\}$ for $P_n$ or from $[n]' $ for $P_n'$).  The group $P_n$ or $P_n'$ acts on the formal complex vector space spanned by the tableaux of shape $v$ by acting as it does for $\Sigma_n$ on the leftmost column and as $H_n$ or $H_n'$ does on the remaining boxes.  The weight of a tableau (in coordinates) is the sum of the weights of its leftmost column and of the remaining columns.  A semistandard tableau of shape $v$ is not required to have the entry in the half-box in its $i$th row less than or equal to the other entries in this row (but apart from this the entries in each row must still increase weakly).

To define the Schur module $E^{v,f}$ or $E^{v,f'}$ in this case we start from the Schur module $E^v$, defined by forming all tableaux $T$ with entries in $[n]''$ or $[n]'$ of shape $\lambda$, supplementing each with a column of half-boxes on the left as for $S_n$, and then moding out by the alternating and exchange relations involving only boxes in $T$.  Now we need additional relations. To define these, let $T$ be a tableau of shape $v$ such that the $i$th entry $a_i$ of the first column (of half-boxes) is $i$ or $\bar i$ as usual.  Write $\bar a_i=\bar i$ if $a_i=i$ and $\bar a_i = i$ if $a_i=\bar i$.  Choose a box $B$ not in the leftmost column of $T$ and let $F$ be the partial filling obtained from $T$ by erasing the entry in $B$.    We then impose the {\sl Pin relation}
\begin{equation}
P_F=\sum_{i=1}^n F_i=0,
\end{equation}  
\noindent for $P_n'$. where $F_i$ is obtained from $F$ by inserting the entry $a_i$ in $B$, changing the $i$th entry of the leftmost column to $\bar a_i$, and leaving all other entries in $F$ unchanged.  For $P_n$ the corresponding Pin relation is 
\begin{equation}
P_F=\sum_{i=0}^n F_i=0,
\end{equation}
\noindent where $F_i$ is defined as above for $i>0$, while $F_0$ is obtained from $F$ by inserting 0 in the box $B$.  The Schur modules $E^{v,f}$ and $E^{v,f'}$ defined by these relations carry actions of $P_n$ and $P_n'$, respectively.   (We will see shortly that the orthogonal relations are consequences of the Pin, alternating, and exchange relations.)  Again we have explicit bases for $E^{v,f}$ and $E^{v,f'}$.  Call a semistandard tableau a {\sl Pin tableau} if it satisfies the symplectic condition and the {\sl spin parity condition} that if the entries $\bar i$ occurs in the leftmost column of the $i$th row and an entry $i$ also occurs in this row, then there is an entry $\bar i$ immediately over the $i$, and similarly with the roles of $i$ and $\bar i$ reversed.   (This condition is a corrected version of rule $R_8$ in \cite{KE83}.)

\begin{theorem}
A basis of $E^{v,f}$ or $E^{v,f'}$ is given by Pin tableaux of shape $v$ using the appropriate set of entries.  These modules are irreducible with highest weight $v$.  The vanishing ideals of the flag varieties of $P_n$ and $P_n'$ are generated by quadratic polynomials corresponding to the alternating, exchange, and Pin relations.
\end{theorem}

\begin{proof}
Suppose first that a semistandard tableau $T$ is not Pin and that a violation of the symplectic or spin parity condition occurs at the $i$th entry of some column $C$, but not at any higher entry of $C$ or at any entry of any column to its left.  If the symplectic condition is violated, then, as noted above, the $j$th entry of $C$ is either $j$ or $\bar j$ for all $j<i$ and it is $i-1$ for $j=i-1$; since the spin parity condition holds above the $i$th entry of $C$, its $j$th entry agrees with the $j$th entry of the first column.  Applying the Pin relation to the filling obtained by erasing the$i$th entry of $C$ and using the alternating relations, we rewrite $T$ as a combination of tableaux for which there is no symplectic or parity violation to the left of $C$ or in $C$ at the $i$th box or above.  Similarly, if instead the spin parity condition is violated at the $i$th box of $C$ but not above it in this column or anywhere in any column to the left of $C$, then applying the Pin relation to the filling obtained by erasing the entry in this box rewrites $T$ as a linear combination of tableaux with no violation of either condition to the left of $C$ or in $C$ at the $i$th box or above.  Iterating this process, we rewrite $T$ as a linear combination of Pin tableaux.  

We prove that the Pin tableaux of shape $v$ are independent and span the irreducible representation $V_v$ by induction on the dimension $m$ of $E$ or $E'$.  The two cases ($m$ odd and $m$ even) are handled as in the orthogonal case; the same branching rule holds.  If $m=2n$ is even, then one defines $e_s$ and $e_d$ as in the orthogonal case and uses the entries $s$ and $d$ in the full boxes in the tableau instead of $n$ and $\bar n$ (but still $n$ or $\bar n$ in the bottom box of the leftmost column of half-boxes;) in any case, the action of $P_{n-1}$ on this entry is trivial.  In any Pin relation, we replace $n,\bar n$ by $(1/2)(s+d),(1/2)(s-d)$, respectively, when inserting entries in full boxes, but continue to use $n$ and $\bar n$ in the half-boxes.  Given a highest weight tableau $T$ for the $P_{n-1}$ action, the entries in the full boxes in the $i$th row are all $i$ if $i<n$, except for bottom entries in their columns, which are allowed to be $d$; in the $n$th row all entries in full boxes are $d$ or $s$.  Applying the Pin relation twice, first to the partial filling obtained by erasing one box in the last row and then again by erasing the box next to it, and using the exchange relations, we find that this tableau equals 0 in $E^{v,f}$ or $E^{v,f'}$ if the entries $d$ and $s$ both appear in the bottom row.  Then $E^{v,f'}$ decomposes in the same way over $P_{n-1}$ as $V_v$ does and the result follows.

With the basis in place, if one applies Pin relations twice in succession to the filling obtained from a tableau $T$ by erasing the entry in one box $B$, and then the filling obtained by erasing the entry in another box $B'$ in its row, and project to the span of the basis elements whose leftmost columns are the same.as that of $T$, then the orthogonal relation corresponding to $B$ and $B'$ follows.  (As noted above, it suffices to derive the orthogonal relations in the case where the boxes $B$ and $B'$ are in the same row.)   With these relations in place, we can now prove the remaining assertion of the theorem as in the symplectic case.
\end{proof}

We conclude by remarking that Ramanathan has shown that the vanishing ideal of the flag variety of any reductive group is generated by quadratic polynomials \cite{R87}.


\begin{thebibliography}{KE83}

\bibitem[B86]{B86} A.\ Berele, \textsl{Construction of Sp-modules by tableaux}, Linear and Multilinear Algebra~{\bf19} (1986), 299--307.


\bibitem[F97]{F97} W.\ Fulton, \textsl{Young tableaux: with applications to representation theory and geometry}, London Math.\ Soc.\ Student Texts~{\bf35}, Cambridge University Press, New York, 1997.

\bibitem[HP06]{HP06} J.-S.\ Huang and P.\ Pand\v zi\' c, \textsl{Dirac Operators and Representation Theory}, Birkh\" auser, Boston, 2006.

\bibitem[KE83]{KE83} R.\ C.\ King and Nahid El-Sharkaway, \textsl{Standard Young tableaux and weight multiplicities of the classical Lie groups}, J.\ Phys.\ A, Math.\ Gen.~{\bf16} (1983), 3153--3177.

\bibitem[P94]{P94} R.\ Proctor, \textsl{Young Tableaux, Gelfand Patterns, and Branching Rules for Classical Groups}, J.\ Al.g~{\bf164} (1994), 299--360.

\bibitem[R87]{R87} A.\ Ramanathan, \textsl{Equations defining Schubert varieties and Frobenius splitting of diagonals}, Pub.\ Math.\ I.H.E.S.~{\bf65} (1987), 61--90.

\bibitem[S90']{S90'} S.\ Sundaram, \textsl{Tableaux in the representation theory of the classical groups}, in \textsl{Invariant Theory and Tableaux}, D.\ Stanton, ed., IMA Series in Pure and Applied Math.~{\bf19}, Springer, 1990, 191--226. 

\bibitem[S90]{S90} S.\ Sundaram, \textsl{Orthogonal tableaux and an insertion algorithm for $SO(2n+1)$}, J.\ Comb.\ Theory, Ser.\  A~{\bf53} (1990), 239--256.

\bibitem[Z73]{Z73} D. P.\ Zhelobenko, \textsl{Compact Lie groups and their representations}, Transl.\ Math.\ Monographs~{\bf40}, A.M.S., Providence, 1973.


\end{thebibliography}
\end{document}